\documentclass[11pt]{amsart}
\usepackage{graphicx} % Required for inserting images
\usepackage{url}
\usepackage{hyperref}
\usepackage{pdfsync}
\usepackage{amssymb}
\usepackage{amsmath}
\usepackage{amsthm}
\usepackage{stmaryrd}
\usepackage{enumitem}
\usepackage{fullpage}	% smaller margins %incompatible with ucthesis
\usepackage{amscd}   	% for commutative diagrams
\usepackage[all]{xy} 		% for complicated commutative diagrams
\usepackage{comment} 	% for \begin{comment} ... \end{comment}
\usepackage{mathrsfs}      % for \mathscr
\usepackage{booktabs}
\usepackage[table,xcdraw]{xcolor}
\usepackage{caption}

\usepackage{booktabs} % For better horizontal lines
\usepackage{tikz}

\usepackage{setspace}
\newtheorem{theorem}{Theorem}[section]
\newtheorem{lemma}[theorem]{Lemma}
\newtheorem{corollary}[theorem]{Corollary}
\newtheorem{proposition}[theorem]{Proposition}

\theoremstyle{definition}
\newtheorem{definition}[theorem]{Definition}

\newcommand{\hideqed}{\renewcommand{\qed}{}} %% to suppress `\qed'

\theoremstyle{remark}
\newtheorem{remark}[theorem]{Remark}

\usetikzlibrary{decorations.pathmorphing}
\DeclareMathOperator{\irr}{irr}
\DeclareMathOperator{\Br}{Br}
\DeclareMathOperator{\Pic}{Pic}
\DeclareMathOperator{\Gal}{Gal}
\DeclareMathOperator{\Proj}{Proj}
\DeclareMathOperator{\HH}{H}

\DeclareMathOperator{\Sym}{Sym}
\DeclareMathOperator{\cores}{cores}
\DeclareMathOperator{\res}{res}
\DeclareMathOperator{\PGL}{PGL}
\DeclareMathOperator{\rank}{rank}
\DeclareMathOperator{\car}{char}
\DeclareMathOperator{\inv}{inv}
\DeclareMathOperator{\Spec}{Spec}
\newcommand{\kbar}{\overline{k}}
\newcommand{\ksep}{k^s}
\newcommand{\defi}[1]{\textsf{#1}} % for defined terms

\newcommand{\C}{{\mathbb C}}
\newcommand{\Q}{{\mathbb Q}}
\newcommand{\F}{{\mathbb F}}
\newcommand{\Z}{{\mathbb Z}}
\newcommand{\PP}{{\mathbb P}}

\newcommand{\p}{{\mathfrak p}}
\newcommand{\q}{{\mathfrak q}}
\title{The degree of irrationality of del Pezzo surfaces}
\author{Adam Logan, Anthony V\'arilly-Alvarado, and David Zureick-Brown}
\address{Department of Mathematics and Statistics, Carleton University,
1125 Colonel By Drive, Ottawa, ON K1S 5B6, Canada}
\email{adam.m.logan@gmail.com}
\urladdr{https://sites.google.com/view/adamlogan/home}
\address{Department of Mathematics, Rice University MS 136, Houston, TX
77005, USA}
\email{av15@math.rice.edu}
\urladdr{http://math.rice.edu/\~{}av15}
\address{Department of Mathematics, Amherst College, 31 Quadrangle Dr., Amherst, MA
01002, USA}
\email{dzureickbrown@amherst.edu}
\urladdr{https://dmzb.github.io/}


\date{October 2025}

\begin{document}

\begin{abstract}
    For an irreducible variety $X$ over a field $k$, the \defi{degree of irrationality} $\irr_k X$ is the minimal degree of a dominant rational map $X \dasharrow \mathbb{P}_k^{\dim X}$. When $X$ is a curve, this is simply the gonality of $X$. 
    We determine the possible degrees of irrationality of del Pezzo surfaces over an assortment of field types: number fields, local fields, finite fields, and arbitrary fields. 
\end{abstract}


\maketitle

%\noindent{\bf Statements and Declarations.} On behalf of all authors, the corresponding author states that there is no conflict of interest.

%%%%%%%%%%%%%%%%%%%%%%%%%%%%%%%%%%%%%%%%%%%%%%%%%%%%%%%%%%%%%%%%%%%%%%%%
\section{Introduction}

For an irreducible variety $X$ over a field $k$, define the \defi{degree of irrationality} $\irr_k X$ to be the minimal degree of a dominant rational map $X \dasharrow \mathbb{P}_k^{\dim X}$. 
It is a $k$-birational invariant of~$X$ that generalizes the $k$-gonality of a curve $C$. 
Writing $g$ for the geometric genus of $C$, if $g = 0$, then either $C \simeq_k \PP^1_k$ and  $\irr_k C = 1$, or $C$ is anticanonically embedded as a conic in $\mathbb{P}^2_k$ and $\irr_k C = 2$. 
If $g = 1$ and $k$ is algebraically closed, then $\irr_k C = 2$, but in general $\irr_k C$ can be arbitrarily large; see~\cite{ClarkL:There-are-genus-one-curves-of-every-index-over-every-infinite-finitely-generated-field}, for example. 
When $g > 1$, gonality is understood via Brill--Noether theory: the degree of irrationality over an algebraically closed field is generically $\lfloor\frac{g+3}{2} \rfloor$, but for general $k$ if $C(k) = \emptyset$ then $\irr_k C$ can be as large as $2g-2$.

When $X$ is a surface, the situation becomes much more involved, even in the case of geometrically rational surfaces. 
By a theorem of Iskovskikh~\cite [Theorem 1]{Iskovskih:Minimal-models-of-rational-surfaces-over-arbitrary-fields}, a smooth projective geometrically rational surface $X$ over $k$ is $k$-birational to either a del Pezzo surface of degree $1 \leq d \leq 9$ or a rational conic bundle. 
If $X_d$ is a del Pezzo surface over an algebraically closed field, then $\irr_k X_d = 1$. 
Over arbitrary $k$, the degree of irrationality can be larger.

\begin{theorem}
    \label{thm:main}
    Let $X_d$ be a del Pezzo surface of degree $d$ over a field $k$. 
    Then $\irr_k X_d \in \{1,2,3,4,6\}$. 
\end{theorem}

For a fixed field and degree only some of the values $\{1,2,3,4,6\}$ are realizable. 
We describe the possibilities over arbitrary fields, number fields, local fields, and finite fields in Table~\ref{table:delPezzos}.

%%%%%%%%%%%%%%%%%%%%%%%%%
\subsection{Related work} 

Previous work on the degree of irrationality has mostly concentrated on the case of an algebraically closed base field $\kbar$, and thus (in the case of surfaces) on nonnegative Kodaira dimension; see the recent survey of Chen and Martin~\cite[\S1]{chen-martin}.
For example, a very general hypersurface $X$ of degree $d$ in $\mathbb{P}_{\kbar}^{n+1}$ such that $d \geq 2n$ has $\irr_{\kbar} X = d-1$ \cite[Theorem A]{bdpelu}. 
For very general hypersurfaces in various kinds of Fano varieties, the degree of irrationality is the degree of an obvious projection; see~\cite{StapletonU:The-degree-of-irrationality-of-hypersurfaces-in-various-Fano-varieties}. 
For certain products of elliptic curves, the problem is studied in~\cite{yoshihara-product}. 
For simple abelian surfaces, the degree of irrationality is at least $3$ by \cite[Theorem~(3,4)]{alzati-pirola}, and at most $4$ by \cite[Corollary~D]{chen-stapleton}.
Deciding between these two values is an interesting problem, studied in~\cite{martin} and elsewhere.  
For torsors under simple abelian surfaces over arbitrary fields, one expects that the index (and thus the degree of irrationality) can be arbitrarily large for reasons similar to the case of genus 1 curves, but this is still unknown.
Returning to algebraically closed fields, for K3 surfaces of low degree the problem is studied in \cite{moretti, moretti-rojas}, and for Enriques surfaces in \cite{oyanedel-ibacache}; for the remaining family of surfaces of Kodaira dimension $0$, namely quotients of abelian surfaces by groups of fixed-point-free automorphisms, the problem is treated in \cite{yoshihara-hyperelliptic}, especially Theorems 2.3, 2.5, 2.8.  
In higher dimension, the degree of irrationality of a Fano threefold that is a hypersurface in weighted projective space is calculated in \cite{cheltsov-park}.
\medskip

%%%%%%%%%%%%%%%%%%%%%
\subsection{Notation}
Throughout, $k$ denotes a field, $\kbar/k$ is a fixed algebraic closure, and $\ksep$ is the separable closure of $k$ in $\kbar$. 
If $X$ is a $k$-scheme and $K$ is a field extension of $k$, we denote the fibered product $X\times_{\Spec k} \Spec K$ by $X_K$.

%%%%%%%%%%%%%%%%%%%%%%%%%%%%%%%%%%%%%%%%%%%%%%%%%%%%%%%%%%%%%%%%%%%%%%%%
\section*{Acknowledgements}
The authors would like to thank the American Institute of Mathematics, where this work was begun during the workshop ``Degree $d$ points on algebraic surfaces" held in March 2024. 
We are grateful to Nathan Chen and Bianca Viray for organizing this workshop. 
We also gratefully acknowledge the support of the Bernoulli Center at the EPFL, where some of this work was done in the context of the program ``Arithmetic geometry of K3 surfaces" in April 2025. 
We thank Dami\'an Gvirtz-Chen for helpful conversations on conics in characteristic~$2$. 
A.\,V-A.\ was partially supported by a National Science Foundation individual grant, No.\ DMS-2302231, while working on this project. 
D.\,Z-B.\ was partially supported by a National Science Foundation individual grant, No.\ DMS-2430098, while working on this project.

%%%%%%%%%%%%%%%%%%%%%%%%%%%%%%%%%%%%%%%%%%%%%%%%%%%%%%%%%%%%%%%%%%%%%%%%
\section{Preliminaries}

%%%%%%%%%%%%%%%%%%%%%%%%%%%%%%%%%%%%
\subsection{Degree of Irrationality}

\begin{lemma}
    \label{lemma:basics} 
    Let $X$ be a smooth and connected variety of dimension $m$ and degree $d$ in $\PP_k^n$ over a field $k$. 
    For all but the first statement, assume that $k$ is infinite.
    \begin{enumerate}[leftmargin=*]
        \item $\irr_k X$ is a $k$-birational invariant.
        \smallskip

        \item The degree of irrationality is at most $d$.
        \smallskip

        \item Suppose that $X$ has an effective $0$-cycle $C_e$ of degree $e < n-m+1$ in general position.  Then the degree of irrationality is at most $d-e$.
        \smallskip

        \item If the degree of irrationality is $d$, then $\displaystyle \bigcup_{[K:k] \le d} X(K)$ is infinite.
    \end{enumerate}
\end{lemma}

\begin{proof}
    First statement: if $X \dasharrow X'$ is a birational equivalence, then composing it with a dominant map $X' \dasharrow \PP_k^m$ of degree $d$ gives such a map for $X$, so $\irr_k X \le \irr_k X'$; the situation is symmetric, so the claim follows.
    Second statement: a general projection from a linear subspace of $\PP^n_k$ of codimension $m+1$ (outside a Zariski closed subset of the Grassmannian) gives a map to $\PP_k^m$ of degree $d$. 
    Since $k$ is infinite and the Grassmannian is rational, such projections exist. 
    Third statement: project away from the linear subspace spanned by $C_e$. 
    For some constant $c$ this gives a map of degree $c$ to a variety of degree at most $(d-e)/c$. 
    Applying the second statement to this variety gives the desired result. 
    Fourth statement: pull back rational points of $\PP_k^m$ by a map of degree~$d$.
\end{proof}

%%%%%%%%%%%%%%%%%%%%%%%%%%%%%%%%%%%%%%%%%%%%%%%%%%%%%%
\subsection{Background Material on del Pezzo Surfaces}

\begin{definition}\label{def:del-pezzo} \cite[Definition 24.2]{Manin-Cubic-Forms}
    A \defi{del Pezzo surface} $X$ over a field $k$ is a smooth projective geometrically integral surface whose anticanonical sheaf is ample. 
\end{definition}

\begin{remark}\label{rem:dp-base-change}
    Let $K/k$ be a field extension.
    Then $X$ is a del Pezzo surface if and only if  $X_K$ is.
\end{remark}

We write $-K_X$ for the divisor class in $\Pic X$ of the anticanonical sheaf on $X$, and denote the intersection pairing on the Picard group of $X$ by $(\ ,\,)\colon \Pic X \times \Pic X \to \Z$. 
The \defi{degree} of $X$ is the intersection number $d := (K_X,K_X)$; it satisfies $1 \leq d \leq 9$. 
If $d \ge 3$ then $-K_{X}$ is not merely ample but very ample: we obtain an embedding $X \hookrightarrow \PP^d_k$ as a smooth surface of degree $d$. 
For $d = 2, 1$ the anticanonical ring $\oplus_{n \ge 0} H^0(-nK)$ is not generated in degree $1$; rather, it requires generators of degree up to $2, 3$ respectively. 
In these cases, a minimal set of generators of $\oplus_{n \ge 0} H^0(-nK)$ gives rise respectively to a surface in $\PP_k(2,1,1,1)$ defined by a polynomial of weighted degree $4$ and a surface in $\PP(3,2,1,1)$ defined by a polynomial of weighted degree $6$. 
See \cite[\S\S\,9.4.3--11]{rational-points-poonen} or~\cite[\S1.5]{VarillyAlvarado:Arithmetic-of-del-Pezzo-surfaces}.

%%%%%%%%%%%%%%
\subsubsection{Classification over separably closed fields}
The following result is classical. 
A collection of points in $\PP^2(k)$ is said to be in \defi{general position} if no $3$ lie on a line, no $6$ lie on a conic, and no $8$ points lie on a cubic singular at one of the points. 

\begin{theorem} 
    \cite[Theorem~1.6]{VarillyAlvarado:Arithmetic-of-del-Pezzo-surfaces}
    \cite[Theorem 9.4.4]{rational-points-poonen}.
    \label{thm:classification-dPs}
    Let $X_d$ be a del Pezzo surface over a field $k$. 
    The base change $X_{d,\ksep}$ is $\ksep$-isomorphic to a blow-up of $\PP^2_{\ksep}$ at $9 - d$ closed points in general position, unless $d = 8$, in which case it is also possible that $X_{d,\ksep}$ is isomorphic to $\PP^1_{\ksep} \times \PP^1_{\ksep}$.
    \qed
\end{theorem}

\begin{remark}
    The general position condition on a collection of $\leq 8$ points in~$\PP^2(k)$ is equivalent to ampleness of the anticanonical sheaf in the blow-up of the collection of points.
\end{remark}

%%%%%%%%%%%%%%
\subsubsection{Exceptional Curves} Birational morphisms $f\colon X \to Y$ between del Pezzo surfaces are mediated by exceptional curves. 
An \defi{exceptional curve} on $X$ is an irreducible curve $C \subset X_{\ksep}$ such that $(C,C) = (C,K_X) = -1$. 
If $X$ contains a $\Gal(\ksep/k)$-stable set of pairwise nonintersecting exceptional curves then this set can be blown down to obtain a del Pezzo surface $Y$ of lower degree~\cite[Corollary~24.5.2]{Manin-Cubic-Forms}. 

Conversely, let $Y$ be a del Pezzo surface of degree $d$ containing a Galois-stable set $S$ of $\ksep$-points $\{P_1,\dots,P_r\}$ with $d>r$. 
Let $X$ be the blowup of $Y$ in $S$.  Inductively define $X_0 = Y \times_k \ksep$ and let $X_i$ be the blowup of $X_{i-1}$ in the inverse image of $P_i$; by abuse of notation, we also denote the inverse image of $P_j$ in $X_i$ as $P_j$ for $j>i$. 
Clearly $X_{\ksep} \cong X_r$.  Suppose that for all $0 \le i < r$ the point $P_i$ does not lie on any exceptional curve of $X_{i-1}$. 
Then $X_r$ is a del Pezzo surface, and hence $X$ is as well.

\begin{remark} 
    The exceptional curves on a blowup of $\PP^2_k$ in $9-d$ points in general position are described in \cite[Theorem 26.2]{Manin-Cubic-Forms}, the result being repeated without proof in \cite[Proposition 9.4.7]{rational-points-poonen}.
\end{remark}

%%%%%%%%%%%%%%
\subsubsection{Rationality in high degree} 
Del Pezzo surfaces need not be rational over their field of definition. 
However, rationality is automatic in high degree whenever the surface has a rational point. 
We use this result several times:

\begin{theorem}[{\cite[p.\ 642]{Iskovskih:Minimal-models-of-rational-surfaces-over-arbitrary-fields},\cite[Theorem~2.1]{VarillyAlvarado:Arithmetic-of-del-Pezzo-surfaces}}]
    \label{thm:high-degree-rationality}
    Let $X_d$ be a del Pezzo surface over a field $k$. 
    If $d \geq 5$ and $X_d(k) \neq \emptyset$, then $X_d$ is $k$-birational to $\PP^2_k$; in particular, $\irr_k X_d = 1$. 
    The hypothesis that $X_d(k) \neq \emptyset$ is automatic when $d = 1$, $5$ or $7$.
    \qed
\end{theorem}

In low degree a similar result often allows one to conclude that the surface is unirational. 

\begin{theorem}[{\cite[Theorem~29.4, Theorem~30.1]{Manin-Cubic-Forms}}] 
    \label{thm:low-degree-unirationality}
    Let $X$ be a del Pezzo surface of degree $2 \le d \le 4$ with a rational point $P$; if $d = 2$ then suppose that $P$ does not lie on any exceptional curve or on the ramification divisor of the anticanonical map $X \to \PP^2_k$. 
    Then $X$ is unirational. 
    In particular, if $k$ is infinite then $X(k)$ is Zariski dense.
    \qed
\end{theorem}

\begin{remark}
    The hypothesis that in degree $2$ the point $P$ not lie on the ramification curve $R$ of $X \to \PP^2_k$ is not stated in~\cite[Theorem~29.4]{Manin-Cubic-Forms}, but it is used implicitly in the proof. 
    The first step of the proof is to blow up the given point to obtain a lower degree del Pezzo surface. 
    When $d=2$ one must avoid $R$; see~\cite[Corollary~2.11]{Salgado-Testa-VA}.
\end{remark}

%%%%%%%%%%%%%%
\subsubsection{Finite Fields}
Although del Pezzo surfaces over arbitrary fields need not carry rational points (unless $d = 1$, $5$, or $7$), the situation over finite fields is different. 
The following fact is a corollary of~\cite[Theorem 23.1]{Manin-Cubic-Forms}.

\begin{lemma}\label{lemma:dps-finite-field-pts}
    Let $X$ be a del Pezzo surface over a finite field $k$.  Then $X(k) \ne \emptyset$. 
\end{lemma}

\begin{proof} 
    We count points by means of the action of Frobenius on \'etale cohomology. 
    The traces of Frobenius on $\HH^0$ and $\HH^4$ contribute $\#k^2$ and $1$ respectively, while $\HH^2(X) = \rank \Pic X_{\bar k}$, so the trace of Frobenius on $\HH^2$ is $\#k$ times the trace of Frobenius acting on $\Pic X_{\bar k}$. 
    It follows that $\#X(k) \equiv 1 \bmod \#k$, whence the claim. 
\end{proof}

%%%%%%%%%%%%%%
\subsubsection{Root systems, Weyl groups, and Galois actions}
\label{sss:Rd}

Within the geometric Picard group $\Pic(X_{d,\kbar}) = \Pic(X_{d,\ksep})$ of a del Pezzo surface $X_d$ of degree $d \leq 6$ there lurks a root system $R_d$ in $K_{X_d}^\perp$, whose type is given by
\begin{equation*}
    R_d = 
    \begin{cases}
        A_2\times A_1 & \text{if } d = 6,\\
        A_4 & \text{if } d = 5,\\
        D_5 & \text{if } d = 4,\\
        E_{9 - d} & \text{if } d = 3, 2, \text{ or }1;
    \end{cases}
\end{equation*}
see~\cite[Proposition~25.2 and Theorem~25.4]{Manin-Cubic-Forms}. 
Since the action of $\Gal(\ksep/k)$ on $\Pic(X_{d,\ksep})$ preserves the canonical class $K_{X_d}$ and the intersection pairing, there is a representation 
\[
    \Gal(\ksep/k) \to O\big(K_{X_d^\perp}\big),
\]
and the image of this map factors through the Weyl group $W(R_d)$ of the root system $R_d$~\cite[Theorem~23.9(ii)]{Manin-Cubic-Forms}, which is a finite group. 

Among other things, this special structure helps identify or compute the Galois cohomology group $\HH^1(\Gal(\ksep/k)),\Pic X_{d,\ksep})$, which is a birational invariant for $X_d$ (see~\S\ref{sss:An-irrationality-criterion} below). 
In particular, when $k$ is a finite field, this group has the form $\HH^1(\langle C_d\rangle,\Z^{9-d})$, where $\langle C_d\rangle < W(R_d)$ is any cyclic subgroup generated by an element in the conjugacy class $C_d \subset W(R_d)$ of the Frobenius action on $\Pic(X_{d,\ksep}) \simeq \Z^{9-d}$.

%%%%%%%%%%%%%%
\subsubsection{An irrationality criterion}
\label{sss:An-irrationality-criterion}
In due course we need to establish the existence of some del Pezzo surfaces over a field $k$ that are not $k$-birational to $\PP^2_k$. 
We use a $k$-birational invariant associated to the Brauer group of the surface, namely, the Galois cohomology group $\HH^1(\Gal(\ksep/k)),\Pic X_{\ksep})$, which is a $k$-birational invariant for a smooth projective geometrically integral variety over a field $k$ (this follows by combining the proof of~\cite[Theorem~29.1]{Manin-Cubic-Forms} with \cite[\S2.A]{colliot-thelene-sansuc-La-descente-II}). 
This cohomology group is trivial for $X = \PP^2_k$, because $\Gal(\ksep/k)$ is a profinite group acting trivially on the free $\Z$-module $\Pic \PP^2_k \simeq \Z$~\cite[Example~4.2.3]{Gille-Szamuely-CSAs}. 
Consequently:

\begin{proposition}
    \label{prop:irrational}
    If $X$ is a smooth projective surface over a field $k$ and $\HH^1(\Gal(\ksep/k),\Pic X_{\ksep}) \neq 0$, then $X$ is not $k$-rational and $\irr_k X > 1$.
    \qed
\end{proposition}

%%%%%%%%%%%%%%%%%%%%%%%%%%%%%%%%%%%%%%%%%%%%%%%%%%%%%%%%%%%%%%%%%%%%%%%%
\section{Proof of Theorem~\ref{thm:main}}
We organize the proof by the degree $d$ of the surface, going from low degree to high degree. 
We determine the possible degrees not only over arbitrary fields but also over number fields, local fields, and finite fields.  
We begin with a well-known lemma. 

\begin{lemma}
    \label{lemma:nonrational-low-degree} 
    For $1 \le d \le 4$ there exists a nonrational del Pezzo surface $X_d$ of degree $d$ over a number field, a finite field, and a local field $k$.
\end{lemma}

\begin{proof} 
    Over number fields, the references in~\cite[p.\ 308]{VarillyAlvarado:Arithmetic-of-del-Pezzo-surfaces} contain examples of minimal del Pezzo surfaces of degrees $d = 1,\dots,4$ with rational points that fail to satisfy weak approximation, and thus cannot be rational~\cite[Lemma~1.2]{VarillyAlvarado:Arithmetic-of-del-Pezzo-surfaces}. 
    If $d$ is fixed and $k$ is finite and sufficiently large, every conjugacy class in the Weyl group $W(R_d)$ occurs as the action of Frobenius on $\Pic X_{d,\ksep} \simeq \Z^{9-d}$ for some $X_d$ defined over $k$~\cite[Corollary~1.8]{bfl}. 
    Now observe that for each $d$, there is at least one conjugacy class $C_d$ in $W(R_d)$ such that $\HH^1(\langle C_d\rangle,\Z^{9-d}) \neq 0$ (e.g., see~\cite{Urabe-Tables} for $d = 2$, $3$). 
    Now apply Proposition~\ref{prop:irrational}. 

    The nonrational surfaces over finite fields can be lifted to a local field while preserving the Galois action on the geometric Picard group. 
    We can thus find surfaces $X_d$ over a local field such that $\HH^1(\Gal(\ksep/k),\Pic X_{d,\ksep}) \ne 0$, and so the surface is not rational, by Proposition~\ref{prop:irrational}.
\end{proof}

\noindent \underline{\defi{Degree 1}}.
The surface $X_1$ is isomorphic to a smooth sextic in the weighted projective space $\PP_k(1,1,2,3) := \Proj k[x,y,z,w]$, and restriction of the projection $\PP_k(1,1,2,3) \dasharrow \PP_k(1,1,2)$ gives a double covering $\pi\colon X_1\to \mathbb{P}_k(1,1,2)$~\cite[\S1]{VarillyAlvarado:Arithmetic-of-del-Pezzo-surfaces}. 
Weighted projective spaces are toric varieties, so they are rational. 
Thus there exists a dominant rational map $X_1 \dasharrow \PP^2_k$ of degree $2$. 
Alternatively, the $k$-basis $\{x^2,xy,y^2,z\}$ of $\Gamma\left(\mathbb{P}_k(1,1,2),\mathscr{O}(2)\right)$ gives an embedding $\mathbb{P}_k(1,1,2) \hookrightarrow\mathbb{P}^3_k = \Proj k[x_0,x_1,x_2,x_3]$ as the singular cone $x_0x_2 - x_1^2 = 0$. 
Composing with projection away from $(1:1:1:1)$ gives a birational map $f\colon \PP_k(1,1,2) \dasharrow \PP^2_k$. The composition $f\circ \pi \colon X_1 \dasharrow \PP^2_k$ witnesses the inequality $\irr_k X_1 \leq 2$. 
By Lemma~\ref{lemma:nonrational-low-degree}, there are del Pezzo surfaces of degree $1$ that are not $k$-rational, so $\irr_k X_1 = 2$ occurs for number fields, finite fields, and local fields. 
\\

\noindent \underline{\defi{Degree 2}}.
The anticanonical map $\phi \colon X_2 \to \PP^2_k$ is a degree 2 cover~\cite[Remark 1.11]{VarillyAlvarado:Arithmetic-of-del-Pezzo-surfaces}, so $\irr_k X_2 \le 2$.  Again, by Lemma~\ref{lemma:nonrational-low-degree} it is not always $1$, and this holds also for number fields, finite fields, and local fields. \\

\noindent \underline{\defi{Degree 3}}. 
The anticanonical map $\phi\colon X_3 \to \PP^3_k$ embeds $X_3$ as a smooth cubic surface in $\PP^3_k$. 
Combining Lemmas~\ref{lemma:dps-finite-field-pts} and~\ref{lemma:nonrational-low-degree}, one deduces that $X_3$ can be nonrational even if $X_3(k) \neq \emptyset$.
In this case we have $\irr_k X_3 = 2$, because projection from the rational point is a map of degree $2$. 
This case occurs for finite fields, number fields, and local fields. 
On the other hand, if $X_3$ has no rational points, it has no quadratic points either. 
Indeed, if $P, P^\sigma$ are points that are conjugate over a separable quadratic extension of $k$, then the line $L$ through $P, P^\sigma$ is $k$-rational. 
Similarly, if $P$ is defined over an inseparable quadratic extension $K/k$, then there is a unique $k$-rational line $L \ni P$. 
If $L \subset X_3$ we have many rational points, and if not the third point of intersection is rational. 
Thus, if $X_3(k) = \emptyset$, we must have $\irr_k X_3 > 2$.  On the other hand, by Lemma~\ref{lemma:basics} we have $\irr_k X_3 \le 3$, so it is exactly $3$. 
There are smooth cubic surfaces over number fields (see, e.g.,~\cite[47.1]{Manin-Cubic-Forms}) with no rational points. 
Likewise, there are smooth cubic surfaces over local fields with no rational points, for example
\begin{equation}
    \label{eq:Pointless-Cubic-Over-Local-Field}
    X_3/\Q_{11}: 11w^3 + x^3 + 8x^2y + 7xy^2 + 10y^3 + 8x^2z + 6xyz + 10y^2z + 8xz^2 + 10yz^2 + 8z^3 = 0
\end{equation}
in $\PP^3_{\Q_{11}}$, so $\irr_k X_3 = 3$ can occur in these cases. 
By Lemma~\ref{lemma:dps-finite-field-pts} we cannot have $\irr_k X_3 = 3$ if $k$ is finite. 

\begin{remark}
    We briefly explain one way to construct smooth cubic surfaces over local fields without a rational point, like~\eqref{eq:Pointless-Cubic-Over-Local-Field}. 
    Let $K$ be a local field with residue field $k$, let $L/K$ be an unramified extension of degree $3$, let $\ell \subset \PP^2(L)$ be a line with no points that reduce to a point of $\PP^2(k)$, and let $C$ be the union of conjugates of $\ell$. 
    Let $S$ be the surface defined by $\pi_K w^3 = f(x,y,z)$, where $\pi_K$ is a uniformizer and $f$ is the equation of $C$. 
    Then $S$ has no $K$-points, because the only $k$-point is $(1:0:0:0)$, the valuation of $\pi_K w^3$ is $1 \bmod 3$, and that of $f(x,y,z)$ is a multiple of $3$.

    Now let $S'$ be a surface defined by an equation $f'$ of the form 
    \[
        -\pi_K w^3 + f(x,y,z) + \sum_{i=0}^2 c_i w^i g_{3-i}(x,y,z),
    \]
    where each $g_{3-i}$ is a homogeneous polynomial of degree $3-i$ with integral coefficients and $v(c_i) \ge \min(1,i)$. 
    Since the equations of $S$ and $S'$ are the same mod $\pi_K$, the $k$-points of $S'$ are the same as those of $S$: that is, the only one is $(1:0:0:0)$.
    Evaluating $f'$ at $(w_0:x_0:y_0:z_0)$, where $v(w_0) = 0$ and the others all have positive valuation, again we see that $v(-\pi_K w_0^3) = 1$, while all the other terms have $v > 1$. 
    Thus there are no $K$-points. 
    Since cubic polynomials of this form are Zariski dense in the $\PP^{19}$ of nonzero cubic polynomials up to scaling, while cubic polynomials defining singular surfaces constitute a proper closed subspace of this space, we obtain smooth cubic surfaces over $K$ with no rational points. 

    The example above was constructed in this way, randomly choosing the line over $\F_{11^3}$ defined by $x + \alpha^{625} y + \alpha^{223} z$ where $\alpha^3 + 2\alpha + 9 = 0$, lifting the union of its $\F_{11}$-conjugates to $\Q$ by choosing the integer in $[0,11)$ representing each element of $\F_{11}$, and adding $11w^3$. 
\end{remark}

\noindent \underline{\defi{Degree 4}}. 
The anticanonical map $\phi\colon X_4 \to \PP^4_k$ embeds $X_4$ as a smooth intersection of two quadrics in $\PP^4_k$. 
By Lemma~\ref{lemma:basics} we know that $\irr_k X_4 \le 4$.

\begin{lemma}
    \label{lem:dp4-rat-line-rat}
    Let $X_4$ be a del Pezzo surface of degree $4$ over $k$ with a $k$-rational exceptional curve~$E$. 
    Then $X_4$ is $k$-rational.  
\end{lemma}

\begin{proof}
    There are $5$ exceptional curves meeting $E$ and they are pairwise disjoint, so they may be blown down to obtain a del Pezzo surface $X_9$ of degree $9$. 
    Since $X_4 \to X_9$ is a morphism (not just a rational map), rational points on $E$ map to rational points on $X_9$, and so $X_9$ is a form of $\PP^2$ with a rational point and is therefore rational (Theorem~\ref{thm:high-degree-rationality}). 
    This also proves that $X_4$ is rational, because $X_4 \to X_9$ is a birational equivalence.
\end{proof}

Suppose that $X_4$ is not $k$-rational. 
If $X_4$ has rational or quadratic points and $k$ is infinite, then by Theorem~\ref{thm:low-degree-unirationality} the set $X_4(K)$ is Zariski dense, where $K$ is either $k$ or a quadratic extension.  

\begin{lemma}
    \label{lemma:dp4-quadratic-points-well-placed}
    Let $X_4$ be a nonrational del Pezzo surface of degree $4$ over an infinite field $k$, let $K/k$ be an extension of degree at most $2$, and suppose that $X_4(K) \neq \emptyset$. 
    Then there exist infinitely many lines $L \subset \PP^4_k$ meeting $X_4$ in a $0$-cycle of degree exactly $2$.
\end{lemma}

\begin{proof}
    Let $X_4 \subset \PP_k^4$ be defined by $Q_1 = Q_2 = 0$.
    Let $L$ be a line meeting $X_4$ in dimension $1$ or in degree greater than $2$.
    Since $L \cap Q_1 \supseteq L \cap X_4$, the intersection $L \cap Q_1$ has the same property, so $L \subset Q_1$; similarly for $Q_2$, so $L \subset X_4$. 
    However, since $X_4(K)$ is Zariski dense by Theorem~\ref{thm:low-degree-unirationality}, there are infinitely many points not lying on a line. 
    Given such a $K$-point $P \in X_4$, there is a unique line $L_P$ defined over $k$ that passes through $P$ and its conjugate (which is equal to $P$ if $K/k$ is inseparable). 
    The above shows that $L_P$ meets $X_4$ in a $0$-cycle of degree $2$.
\end{proof}

Projecting away from a line meeting $X_4$ in a $0$-cycle of degree $2$ gives a map to $\PP^2$ of degree $2$, so $\irr_k X_4 = 2$. 
\smallskip

Could $\irr_k X_4 = 3$ or $4$? 
If there is a map $X_4 \to \PP_k^2$ of degree $3$, then there is an abundance of cubic points; considering the planes or lines spanned by these points, we obtain rational points. 
It is therefore impossible to have $\irr_k X_4 = 3$. 
Over a general field there may be no points on $X_4$ of degree less than $4$~\cite[Theorem 7.6]{CreutzV:Quadratic-points-on-intersections-of-two-quadrics}, in which case $\irr_k X_4 = 4$ by Lemma~\ref{lemma:basics}.  
Creutz and Viray recently proved \cite[Theorem 1]{creutz-viray-no-quadratic-points} that this is possible when $k$ is the rational field $\Q$.
For a local field, \cite[Theorem 1.2 (1)]{CreutzV:Quadratic-points-on-intersections-of-two-quadrics} shows that this cannot occur, and so $\irr_k X_4 \mid 2$. 

If $k$ is finite, then $\irr_k X_4 \mid 2$ as well, but a different argument is needed. 
Let $P \in X_4(k)$.  If $P$ is not on any exceptional curve, then we blow up $P$ to obtain a cubic surface with a rational point, and so $\irr_k X_4 \le 2$. 
If $P$ is on a unique exceptional curve, that curve is defined over $k$, Lemma~\ref{lem:dp4-rat-line-rat} shows that $X_4$ is rational. 
If $P \in E_1 \cap E_2$, where $E_1, E_2$ are exceptional curves, then $E_1 \cup E_2$ is rational, and projecting away from the $\PP^2_k$ gives a conic bundle $X_4 \to \PP^1_k$. 
Let $C_t$ be the generic fibre, viewed as a conic over $k(t)$; then there is a map $C_t \to \PP^1_{k(t)}$ of degree $2$. 
Spreading this out we obtain a rational map of degree $2$ from $X_4$ to a $\PP^1$-bundle over $\PP^1_k$, which is a rational surface. \\

\noindent \underline{\defi{Degree 5}}. 
By Theorem~\ref{thm:high-degree-rationality}, $X_5$ is always $k$-rational, so $\irr_k X_5 = 1$. \\

Before continuing to degree $6$, we recall some basic facts on the change in local invariants of central simple algebras under extension of the base field and corestriction. 

\begin{lemma}
    \label{lemma:brauer-res-cores}
    Let $K/k$ be a finite separable extension and let $A$ be a central simple algebra over~$k$. 
    Let $A'$ be the algebra corresponding to $\cores_{K/k} \res_{K/k} \alpha$, where $\alpha \in \HH^2(\Gal(\ksep/k),({\ksep})^\times)$ is represented by $A$. 
    Then $A'$ represents $[K:k]\, \alpha$.
\end{lemma}

\begin{proof} 
    This is an immediate consequence of the fact that $\cores_{G/H} \circ \res_{G/H}$ is multiplication by $[G:H]$ on cohomology, which is \cite[Chapter~VIII, Proposition~4]{Serre-Local-Fields}.
\end{proof}

\begin{lemma}
    \label{lemma:local-invariant-change} 
    Let $K/k$ be an extension of nonarchimedean local fields and let $\inv_k\colon \Br k \xrightarrow{\sim} \Q/\Z$ and $\inv_K\colon \Br K \xrightarrow{\sim} \Q/\Z$ denote the standard isomorphisms~\cite[Chapter~XIII, Proposition~6]{Serre-Local-Fields}.
    \smallskip
    \begin{enumerate}[leftmargin=*]
        \item Let $A$ be a central simple algebra over $k$. 
        Then $\inv_K(A \otimes_k K) = [K:k] \inv_k(A)$. 
        \smallskip
        \item Let $B$ be a central simple algebra over $K$. 
        Then $\inv_k \cores_{K/k} (B) = \inv_K(B)$, where $\cores_{K/k}(B)$ denotes the central simple algebra associated to the corestriction of the element of the group $\HH^2(\Gal(K^{s}/K),(K^s)^\times)$ corresponding to $B$.
    \end{enumerate}
\end{lemma}

\begin{proof}
    The first statement is~\cite[Chapter~XIII, Proposition~7]{Serre-Local-Fields}. 
    For the second statement, note that the first statement, together with the isomorphism $\Br k \cong \Q/\Z$, implies that the map $\Br k \to \Br K$ taking $A$ to $A \otimes_k K$ is surjective. 
    The result now follows from Lemma~\ref{lemma:brauer-res-cores}.
\end{proof}

\begin{remark} 
    One easily verifies that the same formulas hold for an extension of archimedean local fields, although the invariant maps are no longer isomorphisms. 
\end{remark}

\begin{corollary}
    \label{cor:global-invariant-change} 
    Let $K/k$ be an extension of global fields.
    \smallskip
    \begin{enumerate}[leftmargin=*]
        \item Let $A$ be a central simple algebra over $k$. 
        For every place $\q$ of $K$ lying over the place $\p$ of $k$, the local invariant of $A \otimes_k K$ at $\q$ is $e_\q f_\q$ times that of $A$ at $\p$.
        \smallskip

        \item Let $B$ be a central simple algebra over $K$.
        For every place $\p$ of $k$, the local invariant of $\cores_{K/k} B$ at $\p$ is the sum of the local invariants of $B$ at places of $K$ above $\p$. \qed
    \end{enumerate}
\end{corollary}

\noindent \underline{\defi{Degree 6}}.  
If $X_6(k) \neq \emptyset$ then $X_6$ is $k$-rational, by Theorem~\ref{thm:high-degree-rationality}, so $\irr_k X_6 = 1$ in this case. 
This always holds when $k$ is finite, by Lemma~\ref{lemma:dps-finite-field-pts}.

We show that if $X_6(k) = \emptyset$ then $\irr_k X_6 \in \{2,3,6\}$, and that all three values may occur over number fields while only $2, 3$ occur over local fields. 
To do this, we use results of \cite{blunk}. 
It is well-known that the exceptional curves on $X_6$ form a ``hexagon.'' 
Let $K, L$ be the \'etale algebras of definition of three pairwise disjoint lines and two opposite lines respectively, so that $[K:k] = 2$ and $[L:k] = 3$. 
Let $C_1$ be the set of $K$-isomorphism classes of central simple algebras $B/K$ of rank~$3^2$ such that $B \otimes_K KL$ and $\cores_{K/F} B$ both split, and let $C_2$ be the set of $L$-isomorphism classes of central simple algebras $Q/L$ of rank~$2^2$ such that $Q \otimes_L KL$ and $\cores_{L/F} Q$ both split. 
Then:

\begin{theorem} \cite[Theorem 3.4]{blunk} 
    \label{thm:Blunk}
    There is a bijection between the set of $k$-isomorphism classes of del Pezzo surfaces of degree $6$ over $k$ and triples $(B,Q,KL)$ with $B \in C_1, Q \in C_2$ modulo the equivalence $(B,Q,KL) \sim (B',Q',KL)$ if $B \cong B'$ and $Q \cong Q'$ by isomorphisms that are identified after tensoring with $KL$.
    \qed
\end{theorem}

\begin{corollary} \cite[Corollary 3.5]{blunk} 
    The surface represented by $(B,Q,KL)$ has a rational point if and only if $B$ and $Q$ are both split.
    \qed
\end{corollary}

\begin{corollary} 
    The minimal degree of a point on $X_6$ divides $6$.
\end{corollary}

\begin{proof} 
    The algebras $B$ and $Q$ split over $KL$, so the result is immediate unless $L \cong K' \oplus k$, where $[K':K] = 2$.
    However, in that case the fact that $B \otimes_K KL$ is split implies that $B$ is already split, and so there is a point over the field $K$ (which splits $Q$ by hypothesis).
\end{proof}

For a del Pezzo surface $X_6$ of degree $6$, we write $d_{X_6}$ for the minimal degree of a point.

\begin{proposition}
    \label{prop:deg-6-irr-matches-deg}
    Let $X_6$ be a del Pezzo surface of degree $6$ over a field $k$. 
    Then $\irr_k X_6 = d_{X_6}$.
\end{proposition}

\begin{proof} 
    The case $d_{X_6} = 1$ was already discussed above. 
    Assume then that $d_{X_6} > 1$, which in light of Lemma~\ref{lemma:dps-finite-field-pts} implies that $k$ is infinite.
    By Lemma~\ref{lemma:basics} we have $d_{X_6} \le \irr X_6$, so it suffices to prove the opposite inequality.  
    If $d_{X_6} = 2$ or $3$, then projecting away from a general point of degree~$2$ or~$3$ gives respectively a del Pezzo surface of degree~$4$ or~$3$ with a point of degree~$2$ or~$3$, so by our work on lower degree surfaces we deduce that $\irr X_6 \le d_{X_6}$. 
    Finally, the case $d_{X_6} = 6$ follows from Lemma~\ref{lemma:basics}. 
\end{proof}

With Proposition~\ref{prop:deg-6-irr-matches-deg} in hand, we turn our attention to the possible values of $d_{X_6}$ depending on whether $k$ is a number field or a local field. 

\begin{remark}
    \label{rem:algebras-exist}
    It is known that for the Brauer group of a local or global field the period and index of all Brauer classes are equal: see~\cite[Example 4.1]{abgv}. 
    Thus we may specify central simple algebras simply by giving their invariants.
\end{remark}

\begin{proposition}
    \label{prop:2-3-6-global} 
    Let $k$ be a number field. 
    Then $d_{X_6} \mid 6$, and all possible values occur.
\end{proposition}

\begin{proof}
    We already know that $d_{X_6} = 1$ is possible. 
    By Theorem~\ref{thm:Blunk}, to specify $X_6$, it is enough to specify the \'etale algebras $K$ and $L$ and central simple algebras $B \in C_1$ and $Q \in C_2$. 

    \begin{enumerate}[leftmargin=.65in]
        \item[$d_{X_6} = 2$:] Let $K$ be a quadratic extension of $k$, let $L = k^3$, let $B = M_3(K)$, and let $Q = M_2(k) \times A \times A$, where $A$ is any quaternion algebra split by $K$. 
        \item[$d_{X_6} = 3$:] Let $K = k \times k$, let $L$ be any cubic extension of $k$, let $A$ be a central simple algebra over $k$ of rank~$9$ with local invariants $1/3, 2/3$ at two primes $\p, \q$ inert in $L/k$, and let $B = A \times A^\vee$ and $Q = M_2(L)$.
        \item[$d_{X_6} = 6$:] Let $K$ be a quadratic extension of $k$, let $\p$ be a prime of $k$ that splits in $K$, let $L$ be a Galois cubic extension of $k$ in which $\p$ is inert, and let $\q$ be a prime that splits in $L$ and is inert in $K$. 
        Let $B$ be a central simple algebra with local invariants $1/3, 2/3$ at the two primes of $K$ above $\p$ and $0$ elsewhere, and let $Q$ be a quaternion algebra with local invariants $1/2$ at two of the primes above $\q$ in $L$ and $0$ elsewhere.
        \qed
    \end{enumerate}
    \hideqed
\end{proof}

\begin{proposition}
    \label{prop:not-6-local} 
    Let $k$ be a local field.
    Then $\irr_k X_6 \le 3$, and all such values occur.
\end{proposition}

\begin{proof}
    If $K$ is a field, then $B$ must be split for its corestriction to $k$ to be split; since $Q$ is a quaternion algebra, it splits over a quadratic extension of $k$ not contained in $L$, and so $\irr_k X_6 \mid 2$. 
    If $K$ is not a field, it must be $k \times k$, so $Q \otimes_L KL$ can only be split if $Q$ is already split. 
    Thus $X_6$ has a rational point over any extension where $B$ is split, and in particular $\irr_k X_6 \mid 3$.  For existence, degree $1$ is clear, degree $2$ follows as above, and for degree $3$ take $K = k \times k$, let $L$ be a cubic extension of $k$, let $B = A \times A^\vee$ where $A$ is the division algebra over $k$ of rank~$3^2$ with invariant $1/3$, and let $Q = M_2(L)$. 
\end{proof}
\medskip

\noindent \underline{\defi{Degree 7}}. 
By Theorem~\ref{thm:high-degree-rationality}, $X_7$ is always $k$-rational, so $\irr_k X_7 = 1$. \\

\noindent \underline{\defi{Degree 8}}.  
One case is easy: if $X_8$ is geometrically $\PP^2_{\ksep}$ with one point blown up, then it has a single exceptional divisor $E$, which must be $\Gal(\ksep/k)$-stable and can thus be contracted over $k$ to give a del Pezzo surface of degree $9$ with a $k$-rational point. 
It follows from Lemma~\ref{lemma:basics}(1) and Theorem~\ref{thm:high-degree-rationality} that $\irr X_8 = 1$. %

The case when $X_8$ is geometrically isomorphic to $\PP^1_{\ksep} \times \PP^1_{\ksep}$ is equally simple when $X_8$ has a $k$-rational point, but otherwise it is more difficult. 
Thus we now assume that $X_8$ has no rational points, and in particular that $k$ is infinite by Lemma~\ref{lemma:dps-finite-field-pts}. 
Let $d_{X_8}$ be the minimal degree of an extension where $X_8$ has a rational point. 
We will show that $d_{X_8} = \irr_k X_8 \in \{2,4\}$, and that $4$ is not possible for local fields or number fields. 
To do this, we will use results of Shramov--Vologodsky and Trepalin \cite{shramov2020automorphismspointlesssurfaces,trepalin}. 
We begin by subdividing into two cases depending on the rank of $\Pic X_8$. 

{(1) $\rank \Pic X_8 = 2$.} 
In this case $X_8 \cong C_1 \times C_2$, where $C_1, C_2$ are conics that are unique up to isomorphism and ordering by~\cite[Lemma~7.3]{shramov2020automorphismspointlesssurfaces}. 
If $C_1$ has a rational point but $C_2$ does not, then clearly $\irr X_8 = d_{X_8} = 2$. 
If neither $C_1$ nor $C_2$ has a rational point, we again subdivide into cases. 

{\em Case 1.} 
If there is a separable quadratic extension $K/k$ where both have rational points, then there is a quadratic point in general position, for if the two points lay on the same fibre in either direction that fibre would be rational. 
Thus we may blow up this point and appeal to Proposition~\ref{prop:deg-6-irr-matches-deg} to conclude that $\irr X_8 = d_{X_8} = 2$.  

\begin{remark}
    \label{rem:bad-deg-8-should-happen}
    Consider the conics over $\F_2(t,u,v)$ defined by $tx^2 + y^2 + xz + vz^2$ and $ux^2 + y^2 + xz + vz^2$. 
    Over $\F_2(t,u,\sqrt v)$ they have points such as $(0:\sqrt v:1)$, but we believe that there is no separable quadratic extension where both have points.
    Thus we cannot avoid considering the case that follows. 
\end{remark}

{\em Case 2.} 
Now suppose that there is an inseparable quadratic extension $K/k$ where both $C_1$ and $C_2$ have rational points. 
We will show that generically there is a map $C_1 \times C_2 \to \PP^2_k$ of degree $2$. 
The map will be determined by $K$-rational points but nevertheless defined over $k$; in order to study its degree and image we will pass to $K$, where we may identify $C_1$ and $C_2$ with $\PP^1_K$.

Since the quadratic extension $K/k$ is inseparable, we have $\car K = 2$ and $K = k(\sqrt w)$ for some $w \in k$. 
Let $P \in \PP^2(K) \setminus \PP^2(k)$. 
Permuting the coordinates, we may assume that the last coordinate of $P$ is $1$; the first two cannot both belong to $k$, so say the second does not, being equal to $a+b \sqrt w$ with $b \ne 0$. 
Then the $k$-linear change of coordinates $(x,y,z) \to (x,(y-az)/b,z)$ makes the second coordinate equal to $\sqrt w$, and an obvious change of coordinates makes the first coordinate $0$. 
So if $C$ is a conic having a $K$-point but no $k$-point, then we may assume $C$ to contain $(0:\sqrt w:1)$ and therefore to be defined by an equation of the form $ax^2 + bxy + cxz + y^2 + wz^2 = 0$. 
By rescaling $x$ we may take $b \in \{0,1\}$.  We may parametrize $C$ by inverting the isomorphism $C \to \PP^1_K$ given by $(x:y-\sqrt w z)$.

We now consider a generic situation, where $C_1, C_2$ are defined respectively by 
\[
    a_1x^2 + xy + c_1xz + y^2 + wz^2\quad \text{and}\quad a_2x^2 + xy + c_2xz + y^2 + wz^2,
\]
where $a_0, a_1, c_0, c_1, w$ are independent transcendentals: that is, we take $k = \F_2(a_0,a_1,c_0,c_1,w)$ and $K = k(\sqrt w)$. 
Parametrizing as above and considering the images of the points $(1:0), (0:1), (1:1) \in \PP^1(k)$, we obtain three points on each $C_i$, and hence three points on $C_1 \times C_2$, all defined over~$K$. 
The dimension of the space of $(1,1)$-forms vanishing on all three is $3$-dimensional (each point imposes two $k$-linear conditions and they turn out to be independent). 
Thus we obtain a map $C_1 \times C_2 \to \PP^2_k$ defined over $k$.

To compute its degree, it is easier to extend the base to $K$ and compose with the map $\PP^1_K \times \PP^1_K \to C_1 \times C_2$, which is given by $(2,0)$-forms on the first factor and $(0,2)$-forms on the second. 
The composition $\PP^1_K \times \PP^1_K \to C_1 \times C_2 \to \PP^2_K$ is thus given by $(2,2)$-forms. 
Generically one checks that the images of $x_0 = 0$ and $y_0 = 0$ are distinct curves (in fact lines), so the map is dominant; the degree of the map is the square of the class of a $(2,2)$-form less the degree of the base locus, so $2 \cdot 2 \cdot 2 - 6 = 2$. 
Thus the degree of irrationality is $2$.

This does not prove that the degree of irrationality is always $2$, since it is possible that when one specializes the parameters the images of $x_0 = 0$ and $y_0 = 0$ cease to be distinct curves. 
Though we have not done this, we expect that it is always possible to choose different curves in $\PP^1 \times \PP^1$ or different points of $\PP^1$ in such a way that the degree is shown to be $2$. 
If the degree is not $2$, it must be $4$, because there is an obvious map of degree $4$ given by $C_i \to \PP^1$ on both factors and no map of degree $3$ (such a map would give points of degree $3$ and hence rational points on at least one of the $C_i$).

{\em Case 3.}  
If there is no quadratic extension where $C_1$ and $C_2$ both have points, then $d_{X_8} = 4$. 
Since every conic over $k$ admits a map of degree $2$ to $\PP^1_k$, there is a map of degree $4$ from $X_8$ to $\PP^1_k \times \PP^1_k$, which is a rational surface, and $\irr_k X_8 = 4$ as well (we use Lemma~\ref{lemma:basics}(4) to ensure that $\irr_k X_8 \ge 4$). 
By~\cite[Theorem 1.7 (c)]{trepalin}, this holds if and only if the class in $\Br k$ corresponding to the product of those given by $C_1, C_2$ is not represented by a conic. 
This is not possible for a local field or a number field (see Remark~\ref{rem:algebras-exist}), but for some fields it is. 
In \cite[Lemma 3.3]{trepalin} it is explained how to construct an example.

{(2) $\rank \Pic X_8 = 1$.}  
By \cite[Lemma~7.3 (i)]{shramov2020automorphismspointlesssurfaces}, the surface $X_8$ is isomorphic to the restriction of scalars of a conic $C$ from $L$ to $k$, where $L$ is the field of definition of $\Pic X_8$, a separable quadratic extension of $k$. 
It follows immediately from this result that $X_8 \otimes_k L \cong C \times C^\sigma$, where $\sigma$ is the nontrivial automorphism of $L$ over $k$, since over $L$ the Picard rank is $2$, and if the two conic factors were not conjugate over the separable
extension $L/k$ the surface would not be defined over $k$. 
Trepalin shows in \cite[Corollary~2.2]{trepalin} that, for every extension $K/k$, $X_8(K)$ is empty if and only if $X_8(KL)$ is empty. 
Indeed, this is trivial if $L \subseteq K$, and if not then $X_8 \times_k K$ is still of the same type. 
(Trepalin makes the additional hypothesis that $k$ is perfect, but this is used only to conclude that $L/k$ is separable, which is proved in \cite[Lemma~7.3]{shramov2020automorphismspointlesssurfaces}.)

If $C$ has $L$-points then $X_8$ has $k$-points, so assume we are not in this case.  
If there is a quadratic extension $K/k$ such that $C(KL) \ne \emptyset$, then also $C^\sigma(KL) \ne \emptyset$ and so $X_8(K) \ne \emptyset$.  

First we consider the case of separable $K$. In this case a $K$-point of $X_8$ is again in general position. 
If not, the two geometric points would be in the same fibre for one of the projections, allowing us to distinguish the projections and showing that $L \subseteq K$. 
Since $K, L$ are both quadratic extensions, this would mean that $L = K$ and that $C$ has $L$-points, contrary to our assumption. 
So, as before, we blow it up and obtain a del Pezzo surface of degree $6$ with a quadratic point, so $\irr_k X_8 = 2$. 
And as before we are always in this case when $k$ is a local field or a number field.

If $K$ is inseparable, we proceed as in Case 2 of the situation where $\rank \Pic X_8 = 2$. 
Namely, we define $k = \F_2(t,u,v_0,v_1,w_0,w_1)$ and take $K = k(\sqrt u), L = k(\alpha)$ respectively to be generic separable and inseparable quadratic, so that $\alpha^2 + \alpha + t = 0$. 
As before, we choose $C_1$ to be a generic conic over $L$ with a $KL$-point; the point can be taken to be $(0:\sqrt u:1)$, and $C_1$ is defined by $(v_0 + v_1 \alpha) x^2 + xy + (w_0 + w_1) \alpha xz + y^2 + uz^2$. 
We obtain $C_2$ by Galois conjugation, replacing $\alpha$ by $\alpha + 1$. 
Again by taking three nonreduced points on $C_1 \times C_2$ as base points we obtain a dominant map to $\PP^2_L$ whose degree is $2$, and conclude that generically $\irr_k X_8 = 2$. 
As before, we believe but cannot conclude that this always holds; if not, there is still a map of degree $4$ to the restriction of scalars of $\PP^1_L$, a rational surface, and so $\irr_k X_8 = 4$.

\begin{remark} 
    Though closely related to the previous argument when $K$ was inseparable, this one seems unavoidably different. 
    We cannot specialize from the new situation to the previous one because of the constraints on the conics, nor in the opposite direction because of the enlarged field of definition.
\end{remark}

Otherwise, we must have $d_{X_8} > 2$. 
Choose a separable quadratic polynomial $f$ such that $C(M) \ne \emptyset$, where $M/L$ is the extension defined by $f$. 
Let $M^\sigma$ be the extension of $L$ defined by $f^\sigma$; then $C^\sigma$ has points over $M^\sigma$. 
Let $N$ be the compositum of $M$ and $M^\sigma$; then $N$ is normal and hence Galois, because the compositum of separable extensions is separable. 
The Galois group $\Gal(N/k)$ is contained in $D_4$. It cannot be $C_2 \times C_2$, because then we would have a quadratic extension $K/k$ as in the last paragraph. 
If it is $C_4$, then we consider $M$-orbits of points found by choosing $M$-points of $C, C^\sigma$ not fixed by any nontrivial automorphism of $M$.
Such an orbit can be viewed as a quartic point of $X_8$ over $k$; either it is in general position or it is contained in a $(1,1)$-curve. 
In the first case it can be blown up to obtain a del Pezzo surface of degree $4$. 
In the second case the blowup has a unique curve of self-intersection $-2$, so we can embed it in $\PP^4_k$ with a unique ordinary double point.  
Projecting away shows that $\irr_k X_8 = 2$, a contradiction because $d_{X_8} > 2$. 
We conclude that $d_{X_8} = \irr_k X_8 = 4$.

Finally, if the Galois group is $D_4$, then let $Q$ be a quartic extension of $k$ contained in $N$ and not containing $L$, which exists by Galois theory.
Then $X_8$ has $QL$-points and hence $Q$-points. 
Let $P = \{P_1, P_2, P_3, P_4\}$ be a $Q$-orbit.  
If an element of $Q$ has only two different images in $N$, it must be contained in the quadratic subfield of $Q$, which contradicts a previous hypothesis. 
Thus, when we identify $X_8 \otimes_k N$ with $\PP^1_N \times \PP^1_N$, the $P_i$ have all different images in both factors. 
So the only possible curve of self-intersection less than $-1$ in the blowup of $X_8$ at $P$ is the strict transform of a $(1,1)$-curve containing the entire orbit. 
Thus, the blowup is either a smooth del Pezzo surface of degree $4$ or a singular surface with an ordinary double point. 
As in the previous paragraph, we conclude that the first case must occur and that $d_{X_8} = \irr_k X_8 = 4$. 

\begin{remark}
    \label{rem:dont-know-4} 
    We have not proved that $\irr_k X_8 = 4$ in the case of Picard rank~$1$, since we do not know that there exists a quadratic extension $L/k$ and a conic $C$ over $L$ such that there is no quadratic extension $K/k$ such that $C$ has $KL$-points. 
    We expect that this holds for a generic conic defined by $x^2 - (a+b\sqrt t)y^2 - (c+d\sqrt t)z^2$ over $k_0(a,b,c,d,t,\sqrt{t})$ for any field $k$ with $\car k \ne 2$. 
    In any case, our uncertainty on this point does not change the list of possible degrees.
\end{remark}

\noindent \underline{\defi{Degree 9}}. 
$X_9$ is a Brauer--Severi variety. 
Assume that $X_9$ is not rational: then it comes from a nontrivial cohomology class $\beta \in \HH^1(\Gal(\ksep/k),\PGL_3(\ksep))$, which gives rise to a central simple algebra $A/k$ of rank~$3^2$, which is split by a separable extension $K/k$ of degree $3$. 
Thus $X_9$ admits a $0$-cycle $C_3$ of degree $3$. 
The intersection of $X_9$ with a codimension-$2$ linear subspace $L$ containing $C_3$ is a $0$-cycle $C_9$ of degree $9$ (it cannot be a curve, because every curve on $X_9$ is linearly equivalent to the intersection of a hypersurface with $X_9$).  
Let $C_6$ be the $0$-cycle $C_9 - C_3$. 
We now prove that $L$ can be chosen such that the points of $C_6$ are in linearly general position---in other words, such that the space of linear forms vanishing on $C_6$ has dimension $4$.

First note that the hypothesis that $X_9$ is not rational implies that $k$ is infinite (combine Lemma~\ref{lemma:dps-finite-field-pts} and Theorem~\ref{thm:high-degree-rationality}). 

\begin{lemma}
    \label{lemma:no-3-collinear}
    No three points of $X_9$ are collinear.
\end{lemma}

\begin{proof}
    In view of the absence of lines on $X_9$, this follows from the fact that $X_9$ is defined by quadrics as in the proof of Lemma~\ref{lemma:dp4-quadratic-points-well-placed}.
\end{proof}

Let $G$ be the subvariety of the Grassmannian of $\PP^7$'s in $\PP^9$ containing a given $\PP^2$.
Since $G$ is a rational variety over an infinite field, its rational points are Zariski dense.  
We prove:

\begin{proposition}
    \label{prop:generic-general}
    The subvariety $N$ of $G$ corresponding to subspaces such that the residual intersection $(X_9 \cap L) \setminus C_3$ is not in linearly general position is not Zariski dense.
\end{proposition}

\begin{proof}
    First we note that $N$ is in fact a closed subvariety. 
    Indeed, we have a map $G \to \Sym^6(X_9)$ and $N$ is the inverse image of the locus $L$ of sets of six points not in linearly general position. 
    But $L$ is defined by the vanishing of some minors, so it is closed in $\Sym^6(X_9)$; it follows that $N$ is closed. 
    Because $G$ is irreducible, it suffices to show that $N \ne G$. 
    Again this can be done over the algebraic closure, so we may take $X_9$ to be the $3$-tuple embedding of $\PP^2$ and $C_3$ to be $(1,0,0),(0,1,0)$, $(0,0,1)$ as in Lemma~\ref{lemma:no-3-collinear}.  
    A random choice of codimension-$2$ subscheme cut out by the linear forms pulling back to
    \[
        xy^2 + x^2z - xyz + xz^2 + yz^2,
    x^2y - xy^2 + x^2z + xyz - y^2z + xz^2
    \]
    yields a residual intersection consisting of $6$ points not contained in any $\PP^4$. 
\end{proof}

Now that we know that there is a $0$-cycle of degree $6$ in linearly general position, we define a rational map $X_9 \dasharrow \PP^3_k$ by projecting away from it. 
This map cannot contract any curve, because no $k$-rational curve is contained in a linear subspace of codimension greater than $1$, so the image is a surface.
The base scheme has degree $6$, so the image has degree at most $3$; thus $\irr_k X_9 \le 3$.

Now $X_9$ cannot have any rational points, since a Brauer--Severi variety with a rational point is rational; nor can it have any quadratic points, since $\beta$ restricts to a nontrivial class in $\Br(L)$ for all quadratic extensions $L/k$. 
(We have $\cores_{L/k} \res_{L/k} \beta = 2 \beta$ by Lemma~\ref{lemma:brauer-res-cores}, and this is not $0$ because $\beta$ has order $3$.)
It follows that $\irr_k X_9 \ge 3$; combining with the above, we obtain $\irr_k X_9 = 3$ when $X_9$ is a nontrivial Brauer--Severi variety.
As already pointed out, this cannot be done over finite fields, but over local fields and number fields there is no difficulty.
\qed
\\

We summarize the results in a table. 
The first column gives the possible degrees of a del Pezzo surface and the remaining columns indicate the
possible degrees of irrationality over the indicated types of fields. 

\begin{table}[ht]
\centering
\captionsetup{skip=0.5em}  %
\rowcolors{2}{gray!10}{white} %
\begin{tabular}{c|cccc}
\toprule
\textbf{\ $d$\ } & \textbf{Finite} & \textbf{Local} & \textbf{Number} & \textbf{Arbitrary} \\
\midrule
1 & 1,2 & 1,2 & 1,2 & 1,2 \\
2 & 1,2 & 1,2 & 1,2 & 1,2 \\
3 & 1,2 & 1,2,3 & 1,2,3 & 1,2,3 \\
4 & 1,2 & 1,2 & 1,2,4 & 1,2,4 \\
5 & 1 & 1 & 1 & 1 \\
6 & 1 & 1,2,3 & 1,2,3,6 & 1,2,3,6 \\
7 & 1 & 1 & 1 & 1 \\
8 & 1 & 1,2 & 1,2 & 1,2,4 \\
9 & 1 & 1,3 & 1,3 & 1,3 \\
\bottomrule
\end{tabular}
\caption{The degree of irrationality of del Pezzo surfaces, tabulated by the degree $d$ of the surface for an assortment of field types.}
\label{table:delPezzos}
\end{table}

\newpage
%%%%%%%%%%%%%%%%%%%%%%%%%%%%%%%%%%%%%%%%%%%%%%%%%%%%%%%%%%%%%%%%%%%%%%%%
\section{Potential additional or future work}
\label{S:potential}
In this final section we describe a few varieties for which it seems
interesting to study the degree of irrationality.

%%%%%%%%%%%%%%%%%%%%%%%%%%%%%%%%%%%%
\subsection{Other Rational Surfaces} 
Having studied del Pezzo surfaces, it is natural to ask about the other type of minimal rational surfaces, namely conic bundles. 
If the base is $\PP^1_k$, then either the surface is rational and $\irr_k X = 1$, or it is not, in which case a $2$-to-$1$ cover of $\PP^1(k(t))$ by the generic fibre shows that $\irr_k X = 2$. 
Other cases may be more interesting, however. 
To begin, of the five types of conic bundles in~\cite[Proposition 2.5]{trepalin}, type $1$ is already considered in this paper, types $2-4$ are all birational and type $3$ is considered in this paper, and for the examples of type $5$ constructed in~\cite[Example 2.6]{trepalin}, it is clear that $\irr_k X \mid 4$. 
However, it is not clear to us whether degree $4$ is possible for a number field.

%%%%%%%%%%%%%%%%%%%%%%%%
\subsection{K3 surfaces} 
Continuing with the work on K3 surfaces over $\C$ referred to in the introduction, it would be interesting to study the degree of irrationality for K3 surfaces with varying Picard lattices.
If there is a genus-$1$ fibration with base $\PP^1_k$ and a section or a $2$-section, then the quotient by the negation map gives a cover of degree $2$ of a rational surface. 
Thus attention in this case should be focused on surfaces with small Picard rank. 
Even basic questions may be difficult. 
For example, Moretti shows \cite[Theorem 1.1]{moretti} that $\irr_\C X = 3$ if $X$ is a smooth K3 surface of degree $4, 6, 10$ and Picard number $1$. 
If $X$ is a smooth quartic over $k$ with a point, then clearly $\irr_k X = 3$, and if $X$ has no points then $\irr_k X = 4$. 
What can be said in degree $6$ or $10$ over an arbitrary field? 
One can ask more refined questions about the invariants defined by Moretti as well. 
For example, Moretti~\cite[Corollary 2.8]{moretti} defines a component of the Brill--Noether locus $W_4^2(X)$ (\cite[before Theorem 1.2]{moretti}) for $X$ a smooth K3 surface of degree $8$ that is a $\PP^3$-bundle over the Fourier--Mukai 
partner of~$X$. 
Can it be determined as a Brauer--Severi bundle over an arbitrary field?

%%%%%%%%%%%%%%%%%%%%%%%%%%%%
\subsection{Fano Threefolds}
In this paper we have studied Fano surfaces.
Likewise, one could consider Fano threefolds, the deformation families of which have been classified over $\C$ \cite[Chapter 12]{iskovskikh-prokhorov}.
As a first step one should probably consider those known to be geometrically rational. We obtain forms of such threefolds from embeddings of rational ones by multiples of the anticanonical divisor (just as the general del Pezzo surface of degree $n$ is in the same component of the Hilbert scheme as the anticanonical embedding of $\PP^2$ blown up in $9-n$ points).  For example, let us consider forms of a $(1,1)$-hypersurface in $\PP^2 \times \PP^2$ \cite[\S 12.6, no.~32]{iskovskikh-prokhorov}.  Such a variety is rational.  The anticanonical embedding is a hyperplane section of the Segre embedding of $\PP^2 \times \PP^2$.  If $\ell/k$ is a quadratic extension, then there
is a $k$-form $V$ of $\PP^2 \times \PP^2$ that is isomorphic to $\PP^2 \times \PP^2$ over $\ell$.  A hyperplane section of the Segre embedding of $V$ is
a $k$-form of a $(1,1)$-hypersurface in $\PP^2 \times \PP^2$ but is not
obviously rational.  One might study the degree of irrationality of such a variety.

\bibliographystyle{alpha}
\bibliography{aim-arith-irr}

\newcommand{\etalchar}[1]{$^{#1}$}
\def\cprime{$'$}
\begin{thebibliography}{BDPE{\etalchar{+}}17}

\bibitem[ABGV11]{abgv}
A.~Auel, E.~Brussel, S.~Garibaldi, and U.~Vishne.
\newblock Open problems on central simple algebras.
\newblock {\em Transform. Groups}, 16(1):219--264, 2011.

\bibitem[AP92]{alzati-pirola}
A.~Alzati and G.~P. Pirola.
\newblock On the holomorphic length of a complex projective variety.
\newblock {\em Arch. Math.}, 59(4):398--402, 1992.

\bibitem[BDPE{\etalchar{+}}17]{bdpelu}
F.~Bastianelli, P.~De~Poi, L.~Ein, R.~Lazarsfeld, and B.~Ullery.
\newblock Measures of irrationality for hypersurfaces of large degree.
\newblock {\em Compos. Math.}, 153(11):2368--2393, 2017.

\bibitem[BFL19]{bfl}
B.~Banwait, F.~Fit{\'e}, and D.~Loughran.
\newblock Del {Pezzo} surfaces over finite fields and their {Frobenius} traces.
\newblock {\em Math. Proc. Camb. Philos. Soc.}, 167(1):35--60, 2019.

\bibitem[Blu10]{blunk}
M.~Blunk.
\newblock Del {Pezzo} surfaces of degree 6 over an arbitrary field.
\newblock {\em J. Algebra}, 323(1):42--58, 2010.

\bibitem[CL19]{ClarkL:There-are-genus-one-curves-of-every-index-over-every-infinite-finitely-generated-field}
P.~L. Clark and A.~Lacy.
\newblock There are genus one curves of every index over every infinite,
  finitely generated field.
\newblock {\em J. Reine Angew. Math.}, 749:65--86, 2019.

\bibitem[CM25]{chen-martin}
N.~Chen and O.~Martin.
\newblock A primer on measures of irrationality.
\newblock Preprint, {arXiv}:2509.03783 [math{NT}] (2025), 2025.

\bibitem[CP25]{cheltsov-park}
I.~Cheltsov and J.~Park.
\newblock Degree of irrationality of {F}ano threefold hypersurfaces.
\newblock In {\em Perspectives on {F}our {D}ecades of {A}lgebraic {G}eometry,
  {V}olume 1}, volume 351 of {\em Progr. Math.}, pages 127--142.
  Birkh\"{a}user/Springer, Cham, 2025.

\bibitem[CS20]{chen-stapleton}
N.~Chen and D.~Stapleton.
\newblock Fano hypersurfaces with arbitrarily large degrees of irrationality.
\newblock {\em Forum Math. Sigma}, 8:12, 2020.
\newblock Id/No e24.

\bibitem[CTS87]{colliot-thelene-sansuc-La-descente-II}
J.-L. Colliot-Th\'el\`ene and J.-J. Sansuc.
\newblock La descente sur les vari\'et\'es rationnelles. {II}.
\newblock {\em Duke Math. J.}, 54(2):375--492, 1987.

\bibitem[CV23]{CreutzV:Quadratic-points-on-intersections-of-two-quadrics}
B.~Creutz and B.~Viray.
\newblock Quadratic points on intersections of two quadrics.
\newblock {\em Algebra Number Theory}, 17(8):1411--1452, 2023.

\bibitem[CV24]{creutz-viray-no-quadratic-points}
B.~Creutz and B.~Viray.
\newblock Quartic del {P}ezzo surfaces without quadratic points.
\newblock Preprint, {arXiv}:2408.08436 [math{NT}] (2024), 2024.

\bibitem[GS17]{Gille-Szamuely-CSAs}
P.~Gille and T.~Szamuely.
\newblock {\em Central simple algebras and {G}alois cohomology}, volume 165 of
  {\em Cambridge Studies in Advanced Mathematics}.
\newblock Cambridge University Press, Cambridge, second edition, 2017.

\bibitem[IP99]{iskovskikh-prokhorov}
V.~A. Iskovskikh and Yu.~G. Prokhorov.
\newblock Fano varieties.
\newblock In {\em Algebraic geometry V: Fano varieties. Transl. from the
  Russian by Yu. G. Prokhorov and S. Tregub}, pages 1--245. Berlin: Springer,
  1999.

\bibitem[Isk79]{Iskovskih:Minimal-models-of-rational-surfaces-over-arbitrary-fields}
V.~A. Iskovskih.
\newblock Minimal models of rational surfaces over arbitrary fields.
\newblock {\em Izv. Akad. Nauk SSSR Ser. Mat.}, 43(1):19--43, 237, 1979.

\bibitem[Man86]{Manin-Cubic-Forms}
Yu.~I. Manin.
\newblock {\em Cubic forms}, volume~4 of {\em North-Holland Mathematical
  Library}.
\newblock North-Holland Publishing Co., Amsterdam, second edition, 1986.
\newblock Algebra, geometry, arithmetic, Translated from the Russian by M.
  Hazewinkel.

\bibitem[Mar22]{martin}
O.~Martin.
\newblock The degree of irrationality of most abelian surfaces is 4.
\newblock {\em Ann. Sci. {\'E}c. Norm. Sup{\'e}r. (4)}, 55(2):569--574, 2022.

\bibitem[Mor23]{moretti}
F.~Moretti.
\newblock The polarized degree of irrationality of ${K3}$ surfaces.
\newblock Preprint, {arXiv}:2303.07289 (2023), 2023.

\bibitem[MR24]{moretti-rojas}
F.~Moretti and A.~Rojas.
\newblock On the degree of irrationality of low genus ${K3}$ surfaces.
\newblock Preprint, {arXiv}:2401.03821 [math.{AG}] (2024), 2024.

\bibitem[OI24]{oyanedel-ibacache}
E.~R. Oyanedel~{I}bacache.
\newblock Degree of irrationality and vector bundles on {K}3 and {E}nriques
  surfaces.
\newblock Master's thesis, Universidad {T}\'ecnica Federico Santa Mar\'\i a,
  2024.
\newblock Available for download at
  \url{https://pmontero.mat.utfsm.cl/pdf_memorias/Memoria_Emilio_Oyanedel.pdf}.

\bibitem[Poo17]{rational-points-poonen}
B.~Poonen.
\newblock {\em Rational points on varieties}, volume 186 of {\em Grad. Stud.
  Math.}
\newblock Providence, RI: American Mathematical Society (AMS), 2017.

\bibitem[Ser79]{Serre-Local-Fields}
J.-P. Serre.
\newblock {\em Local fields}, volume~67 of {\em Graduate Texts in Mathematics}.
\newblock Springer-Verlag, New York-Berlin, 1979.
\newblock Translated from the French by Marvin Jay Greenberg.

\bibitem[STVA14]{Salgado-Testa-VA}
C.~Salgado, D.~Testa, and A.~V\'arilly-Alvarado.
\newblock On the unirationality of del {P}ezzo surfaces of degree 2.
\newblock {\em J. Lond. Math. Soc. (2)}, 90(1):121--139, 2014.

\bibitem[SU20]{StapletonU:The-degree-of-irrationality-of-hypersurfaces-in-various-Fano-varieties}
D.~Stapleton and B.~Ullery.
\newblock The degree of irrationality of hypersurfaces in various {F}ano
  varieties.
\newblock {\em Manuscripta Math.}, 161(3-4):377--408, 2020.

\bibitem[SV20]{shramov2020automorphismspointlesssurfaces}
C.~Shramov and V.~Vologodsky.
\newblock Automorphisms of pointless surfaces, 2020.

\bibitem[Tre23]{trepalin}
A.~Trepalin.
\newblock Birational classification of pointless del {Pezzo} surfaces of degree
  8.
\newblock {\em Eur. J. Math.}, 9(1):21, 2023.
\newblock Id/No 2.

\bibitem[Ura96]{Urabe-Tables}
T.~Urabe.
\newblock Supplement to calculation of manin's invariant for del pezzo
  surfaces.
\newblock {\em Mathematics of Computation}, 65(213):S15--S23, 1996.

\bibitem[VA13]{VarillyAlvarado:Arithmetic-of-del-Pezzo-surfaces}
A.~V\'{a}rilly-Alvarado.
\newblock Arithmetic of del {P}ezzo surfaces.
\newblock In {\em Birational geometry, rational curves, and arithmetic}, Simons
  Symp., pages 293--319. Springer, Cham, 2013.

\bibitem[Yos96]{yoshihara-product}
H.~Yoshihara.
\newblock Degree of irrationality of a product of two elliptic curves.
\newblock {\em Proc. Am. Math. Soc.}, 124(5):1371--1375, 1996.

\bibitem[Yos00]{yoshihara-hyperelliptic}
H.~Yoshihara.
\newblock Degree of irrationality of hyperelliptic surfaces.
\newblock {\em Algebra Colloq.}, 7(3):319--328, 2000.

\end{thebibliography}
\end{document}